\documentclass{amsart}
\usepackage{amsthm,xcolor, bm}
\usepackage{latexsym,amssymb,amsfonts,amsmath,graphicx, url,mathtools}
\usepackage[mathscr]{eucal}
\usepackage[vcentermath,enableskew]{youngtab}
\usepackage{color}
\usepackage[noadjust]{cite} 
\usepackage{tikz,hyperref}
\usepackage[margin=1in]{geometry}
\usepackage{enumerate,textcomp, wasysym, enumitem}

\newtheorem{theorem}{Theorem}[section]

\newtheorem{lemma}[theorem]{Lemma}
\newtheorem{conjecture}[theorem]{Conjecture}
\newtheorem{corollary}[theorem]{Corollary}
\newtheorem{problem}[theorem]{Problem}

\theoremstyle{definition}
\newtheorem{definition}[theorem]{Definition}
\newtheorem{example}[theorem]{Example}

\newtheorem*{theorem1*}{Theorem \ref{thm:main}}
\newtheorem*{theorem2*}{Theorem \ref{thm:cw}}
\newtheorem*{theorem3*}{Theorem \ref{cor:main}}
\newtheorem*{theorem4*}{Theorem \ref{thm:gen}}
\newtheorem*{conjecture1*}{Conjecture \ref{conj:gao}}
\newtheoremstyle{named}{}{}{\itshape}{}{\bfseries}{.}{.5em}{#1 \thmnote{#3}}
\theoremstyle{named}

\newcommand*{\C}{\mathbb{C}}

\def\S{\mathfrak{S}}

\def\Z{\mathbb{Z}}

\def\hd{{\widehat{D}}}
\def\hc{{\widehat{C}}}

\newcommand*{\Schub}{\mathfrak{S}}

\DeclareMathOperator*{\Span}{Span}

\DeclareMathOperator*{\purple}{Purple}
\DeclareMathOperator*{\Purple}{\mathbf{P}}

\DeclareMathOperator*{\perm}{perm}

\title{Inclusion-exclusion on Schubert polynomials}

\author{Karola M\'esz\'aros}
\address{Karola M\'esz\'aros, Department of Mathematics, Cornell University, Ithaca, NY 14853. \newline\textup{karola@math.cornell.edu}
}

\author{Arthur Tanjaya}
\address{Arthur Tanjaya, Department of Mathematics, Cornell University, Ithaca NY 14853.  \newline\textup{amt333@cornell.edu}
}

\thanks{Karola M\'esz\'aros is partially  supported by  CAREER NSF Grant DMS-1847284.}

\subjclass[2010]{05E05}
\begin{document}

	\begin{abstract}   We prove that an inclusion-exclusion inspired  expression of Schubert polynomials of permutations that avoid  the patterns $1432$ and $1423$ is nonnegative. Our theorem implies a partial affirmative answer to a   recent conjecture of Yibo Gao about principal specializations of  Schubert polynomials.  We propose a general framework for finding   inclusion-exclusion inspired  expression of Schubert polynomials of all permutations.   \end{abstract}

	\maketitle

	\section{Introduction}
	Schubert polynomials, introduced  by Lascoux and Sch\"utzenberger in \cite{LS1}, represent cohomology classes of Schubert cycles in the flag variety. They are also multidegrees of matrix Schubert varieties \cite{multidegree} and  wield an impressive collection of   combinatorial formulas \cite{laddermoves, BJS, FK1993, nilcoxeter, thomas, lenart, manivel, prismtableaux}. Yet, only recently have their supports been established as integer points of generalized permutahedra \cite{FMS, MTY}. There has also been several exciting recent developments about the coefficients of Schubert polynomials: (1) they are  known to be log-concave along root directions in their Newton polytopes \cite{june}; (2) the set of permutations whose Schubert polynomials have all their coefficients less than or equal to a fixed integer $m$ is closed under pattern containment \cite{zeroone}. Recall that $\pi=\pi_1\ldots \pi_k \in S_k$ is a pattern of $\sigma=\sigma_1\ldots \sigma_n \in S_n$ if and only if there are indices $1\leq i_1<i_2<\cdots<i_k\leq n$ so that the relative order of  
	$\pi_1, \ldots, \pi_k$ and of $\sigma_{i_1}, \ldots, \sigma_{i_k}$ are the same.  

	\subsection{Nonnegative linear combinations of Schubert polynomials with monomial coefficients.} In this paper we investigate nonnegativity properties of linear combinations   of Schubert polynomials with monomial coefficients in $\Z[x_1, \ldots, x_n]$ associated to patterns of a fixed permutation.  A first step in this direction is a recent result by Fink, St. Dizier and the first author of the present paper: 	 

		\begin{theorem} \cite[Theorem 1.2]{zeroone}  \label{thm:fms} 		Fix  $\sigma \in S_n$ and let $\pi \in S_{n-1}$ be the pattern of $\sigma$ with Rothe diagram $D(\pi)$ obtained by removing row $k$  and column $\sigma_k$ from $D(\sigma)$. Then
		\begin{align} \label{eq:?}
			\mathfrak{S}_{\sigma}(x_1, \ldots, x_n)-M_{\sigma, \pi}(x_1, \ldots, x_n) \mathfrak{S}_{\pi}(x_{1}, \ldots, \widehat{x_k}, \ldots, x_{n}) \in \mathbb{Z}_{\geq 0}[x_1, \ldots, x_n]		\end{align}  
		where   
		\[M_{\sigma, \pi}(x_1, \ldots, x_n) = \left(\prod_{(k,i)\in D(\sigma)}{x_k}\right)\left(\prod_{(i,\sigma_k)\in D(\sigma)}{x_i} \right).\]
		\end{theorem}

		In particular, Theorem \ref{thm:fms} implies that the set of permutations whose Schubert polynomials have all their coefficients less than or equal to a fixed integer $m$ is closed under pattern containment.

\medskip

		The first result of this paper is a broad extension of Theorem \ref{thm:fms} for $1432$ and $1423$ avoiding permutations:   

\begin{theorem} \label{thm:main}
Let $w \in S_n$ be a $1432$ and $1423$ avoiding permutation and let $u$ be a subword of $w$. Then
\begin{equation} \label{eqn:f_w1}
\sum_{u \le v \le w} (-1)^{\left|w\right|-\left|v\right|} M_{w, v} \Schub_{\perm(v)}(\mathbf{x}_{w^{-1}(v)}) \in \Z_{\ge0}[x_1, \dots, x_n],
\end{equation}
where
\begin{equation*}
M_{w, v} \coloneqq \prod_{(i, j) \in D(w) \setminus \widehat{D(w)}_v} x_i.
\end{equation*}
\end{theorem}

In Theorem \ref{thm:main} we use the relation of containment on words: for words $u, v$, we say $u \le v$ if $u$ occurs as a subword in $v$. Moreover, for a word $v$ of length $n$, $\pi = \perm(v)$ is the permutation in $S_n$ such that the relative order of $\pi_1, \dots, \pi_n$ and of $v_1, \dots, v_n$ are the same.
For these and  other definitions used in Theorems \ref{thm:fms} and \ref{thm:main} see   Sections \ref{sec:background} and \ref{sec:main} which lay them out in detail. Here we give an example of Theorem \ref{thm:main} for illustration.
For $w = 1342$ and $u = 42$ we have
$\{ v \mid u \le v \le w \} = \{ 1342, 142, 342, 42 \},$
so the alternating sum in (\ref{eqn:f_w1}) becomes
\begin{align*}
&M_{w, 1342}\Schub_{1342}(x_1,x_2, x_3, x_4) - M_{w, 142}\Schub_{132}(x_1, x_3, x_4) - M_{w, 342}\Schub_{231}(x_2, x_3, x_4) + M_{w, 42}\Schub_{21}(x_3, x_4) \\
&\quad = 1 \cdot (x_1x_2 + x_1x_3 + x_2x_3) - x_2 \cdot (x_1 + x_3) - 1 \cdot (x_2 x_3) + x_2 \cdot (x_3) \\
&\quad = x_1x_3,
\end{align*}
which indeed has nonnegative coefficients. See Figure \ref{fig2} for an illustration.  
\medskip

	An immediate corollary of  Theorem \ref{thm:main}  is the following theorem: 
	\begin{theorem} \label{cor:main}
		Let $w \in S_n$ be a $1432$ and $1423$ avoiding permutation. If $u$ is a subword of $w$, then
		\begin{equation} \label{eqn:cw}
			\sum_{u \le v \le w} (-1)^{\left|w\right|-\left|v\right|} \Schub_{\textup{perm}(v)}({\bf1})\geq 0,
		\end{equation}
		where $\Schub_{\textup{perm}(v)}({\bf1})$ denotes the value of the Schubert polynomial $\Schub_{\textup{perm}(v)}$ with all its variables set to $1$.
	\end{theorem}

		Theorem \ref{cor:main} is closely related to a recent conjecture of Gao \cite[Conjecture 3.2]{Gao} regarding the principal specialization of Schubert polynomials as we now explain. We also conjecture (Conjecture \ref{conj:1.3}) in Section \ref{sec:conj} that Theorem \ref{cor:main} holds for all permutations $w \in S_n$. 

		\subsection{Principal specializations of Schubert polynomials.} Macdonald \cite[Eq. 6.11]{macdonald} famously expressed the principal specialization $\Schub_{\sigma}(\bf{1})$ of the Schubert polynomial $\Schub_{\sigma}$  in terms of the reduced words of $\sigma$. Fomin and Kirillov \cite{reduced} placed this expression in the context of plane partitions for dominant permutations, while after two decades Billey et al. \cite{bijmac} provided a combinatorial proof. In 2017, Stanley \cite{she} considered the asymptotics of   $\Schub_{\sigma}(\bf{1})$  as well as the role pattern containment plays in its value. The asymptotics question was partially answered by Morales, Pak and Panova \cite{pak}, while the pattern avoidance question  inspired Weigandt \cite{132} and Gao \cite{Gao}, among others, to seek an understanding of $\Schub_{\sigma}(\bf{1})$ in terms of the permutation patterns of $\sigma$. Weigandt showed that $\Schub_{\sigma}({\bf 1})\geq 1+p_{132}(\sigma)$, where  $p_{\pi}(\sigma)$ is the number of patterns $\pi$ in the permutation $\sigma$, while Gao improved this to $\Schub_{\sigma}({\bf 1})\geq 1+p_{132}(\sigma)+p_{1432}(\sigma)$. Gao conjectured that there exist nonnegative integers $c_w$, for $w \in S_{\infty}$, such that    $$\Schub_{\sigma}({\bf 1})=\sum_{\pi \in S_{\infty}} c_{\pi}p_{\pi}(\sigma).$$ 

Equivalently:

	\begin{conjecture}(\cite[Conjecture 3.2]{Gao})\label{conj:gao} There exist nonnegative integers $c_w$, for $w \in S_{\infty}$, such that    $$\Schub_w({\bf 1})=\sum_{v \leq w} c_{{\rm perm}(v)},$$ where  
$v \leq w$ denotes that $v$ occurs as a subword in $w$. 
	\end{conjecture}

	It follows readily via inclusion-exclusion that 
	for $w \in S_{\infty}$:   \begin{equation} \label{eq:cw} c_w=\sum_{v \le w} (-1)^{\left|w\right|-\left|v\right|} \Schub_{\perm(v)}(\bf{1}).\end{equation} Thus, Theorem \ref{cor:main} settles Gao's conjecture \ref{conj:gao} for $1432$ and $1423$ avoiding permutations  $w \in S_{\infty}$ when we specialize it to the empty word $u=()$. Moreover, we also 
	provide a combinatorial interpretation of the numbers $c_w$ for $1432$ and $1423$ avoiding permutations  $w \in S_{\infty}$:

	\begin{theorem} \label{thm:cw}  For $1432$ and $1423$ avoiding permutations  $w \in S_{\infty}$ the value of $c_w$ is the number of diagrams $C \le D(w)$ that cannot be written as $\widehat{C}_\textup{aug}$ for some $\widehat{C} \le \widehat{D(w)}$.  	\end{theorem}

	See  Section \ref{sec:gao} for more details.

	\medskip

	\subsection{Extending Theorems \ref{thm:fms} \& \ref{thm:main}.} Both Theorem \ref{cor:main} and Theorem \ref{thm:cw} are byproducts of our main Theorem \ref{thm:main}. It is thus most natural to ask in what generality  Theorem \ref{thm:main} holds. While Theorem \ref{cor:main} is conjectured by Gao to hold for all permutations, Theorems \ref{thm:main} and 	\ref{thm:cw} as stated do not. Theorem \ref{thm:main} fails already for $w=1432$. However, the reason it fails leads to other possibilities: the  monomials $M_{w,v}$ we used to formulate Theorem \ref{thm:main} are inspired by Theorem \ref{thm:fms} and  are one of many choices we might have made.   While Fink, M\'esz\'aros, and St.~Dizier  \cite{zeroone} only constructed one monomial $M_{\sigma, \pi}$ for the pair of permutations $(\sigma, \pi)$ in Theorem \ref{thm:fms},   there is a family of  monomials each of which would make \eqref{eq:?} true. We are lead to wonder whether  for an appropriate choice of such monomials Theorem \ref{thm:main} could be generalized to any permutation. We take the first step towards this goal  via the following generalization of   Theorem \ref{thm:fms} showing that a family of monomials, including $M_{\sigma, \pi}$ could work:

	\begin{theorem} \label{thm:gen}
	Fix  $\sigma \in S_n$ and let $\pi \in S_{n-1}$ be the pattern of $\sigma$ with Rothe diagram $D(\pi)$ obtained by removing row $k$  and column $\sigma_k$ from $D(\sigma)$. 
If there is some diagram $K \in \Purple_{k, \sigma_k}(D(\sigma))$ such that
\[
M(x_1, \dots, x_n) = \prod_{(i, j) \in K} x_i,
\]
then
 \[
\mathfrak{S}_{\sigma}(x_1, \ldots, x_n)-M(x_1, \ldots, x_n) \mathfrak{S}_{\pi}(x_{1}, \ldots, \widehat{x_k}, \ldots, x_{n}) \in \mathbb{Z}_{\geq 0}[x_1, \ldots, x_n].\]\end{theorem}

 See Section \ref{sec:gen} for the definition of the set of diagrams $\Purple_{k, \sigma_k}(D(\sigma))$ used in the statement of Theorem \ref{thm:gen} above and Section \ref{sec:conj} for a discussion of how  Theorem \ref{thm:gen} could be used to generalize Theorem \ref{thm:main} as well as Conjecture \ref{conj:16} examining the strength of  Theorem \ref{thm:gen}. 

	\subsection*{Outline of this paper}
	Section~\ref{sec:background}  lays out the general background on Schubert polynomials that we rely on. Section \ref{sec:main} contains the setup and proofs of Theorem \ref{thm:main}, \ref{cor:main} and  \ref{thm:cw}.  Section \ref{sec:gen} provides a proof of Theorem \ref{thm:gen} and its generalization Theorem \ref{thm:gen1}, while Section \ref{sec:conj} concludes with conjectures and open problems.

	\section{Background on Schubert polynomials}
	\label{sec:background}

	Schubert polynomials were originally defined via divided difference operators. We will instead define them as dual chatacters of flagged Weyl modules for Rothe diagrams. This section follows the exposition of \cite{FMS, zeroone}. 

	\subsection{Definition of dual characters of flagged Weyl modules.}
	A \emph{diagram} is a sequence $D = (C_1, C_2, \ldots,$ $C_n)$ of finite subsets of $[n]$, called the \emph{columns} of $D$. We interchangeably think of $D$ as a collection of boxes $(i,j)$ in a grid, viewing an element $i\in C_j$ as a box in row $i$ and column $j$ of the grid. When we draw diagrams, we read the indices as in a matrix: $i$ increases top-to-bottom and $j$ increases left-to-right. 

	The \emph{Rothe diagram} $D(w)$ of a permutation $w\in S_n$ is the diagram
	\[ D(w)=\{(i,j)\in [n]\times [n] \mid i<(w^{-1})_j\mbox{ and } j<w_i \}. \]
	Note that Rothe diagrams have the \emph{northwest property}: If $(r,c'),(r',c)\in D(w)$ with $r<r'$ and $c<c'$, then $(r,c)\in D(w)$.

	Let $G=\mathrm{GL}(n,\mathbb{C})$ be the group of $n\times n$ invertible matrices over $\mathbb{C}$ and $B$ be the subgroup of $G$ consisting of the $n\times n$ upper-triangular matrices. The flagged Weyl module is a representation $\mathcal{M}_D$ of $B$ associated to a diagram $D$. The dual character of $\mathcal{M}_D$ has been shown in certain cases to be a Schubert polynomial \cite{KP} or a key polynomial \cite{flaggedLRrule}. We will use the construction of $\mathcal{M}_D$ in terms of determinants given in \cite{magyar}.

	Denote by $Y$ the $n\times n$ matrix with indeterminates $y_{ij}$ in the upper-triangular positions $i\leq j$ and zeros elsewhere. Let $\mathbb{C}[Y]$ be the polynomial ring in the indeterminates $\{y_{ij}\}_{i\leq j}$. Note that $B$ acts on $\mathbb{C}[Y]$ on the right via left translation: if $f(Y)\in \mathbb{C}[Y]$, then a matrix $b\in B$ acts on $f$ by $f(Y)\cdot b=f(b^{-1}Y)$. For any $R,S\subseteq [n]$, let $Y_S^R$ be the submatrix of $Y$ obtained by restricting to rows $R$ and columns $S$.

	For $R,S\subseteq [n]$, we say $R\leq S$ if $\#R=\#S$ and the $k$\/th least element of $R$ does not exceed the $k$\/th least element of $S$ for each $k$. For any diagrams $C=(C_1,\ldots, C_n)$ and $D=(D_1,\ldots, D_n)$, we say $C\leq D$ if $C_j\leq D_j$ for all $j\in[n]$.

	\begin{definition}
		For a diagram $D=(D_1,\ldots, D_n)$, the \emph{flagged Weyl module} $\mathcal{M}_D$ is defined by
		\[\mathcal{M}_D=\mathrm{Span}_\mathbb{C}\left\{\prod_{j=1}^{n}\det\left(Y_{D_j}^{C_j}\right)\ \middle|\   C\leq D \right\}. \]
		$\mathcal{M}_D$ is a $B$-module with the action inherited from the action of $B$ on $\mathbb{C}[Y]$. 
	\end{definition}
	Note that since $Y$ is upper-triangular, the condition $C\leq D$ is technically unnecessary since $\det\left(Y_{D_j}^{C_j}\right)=0$ unless $C_j\leq D_j$. Conversely, if $C_j\leq D_j$, then $\det\left(Y_{D_j}^{C_j}\right)\neq 0$. 

	For any $B$-module $N$, the \emph{character} of $N$ is defined by $\mathrm{char}(N)(x_1,\ldots,x_n)=\mathrm{tr}\left(X:N\to N\right)$, where $X$ is the diagonal matrix $\mathrm{diag}(x_1,x_2,\ldots,x_n)$ with diagonal entries $x_1,\ldots,x_n$, and $X$ is viewed as a linear map from $N$ to $N$ via the $B$-action. Define the \emph{dual character} of $N$ to be the character of the dual module $N^*$:
	\begin{align*}
		\mathrm{char}^*(N)(x_1,\ldots,x_n)&=\mathrm{tr}\left(X:N^*\to N^*\right) \\
		&=\mathrm{char}(N)(x_1^{-1},\ldots,x_n^{-1}).
	\end{align*}

	\begin{definition}
		For a diagram $D\subseteq [n]\times [n]$, let $\chi_D=\chi_D(x_1,\ldots,x_n)$ be the dual character 
		\[\chi_D=\mathrm{char}^*\mathcal{M}_D. \] 
	\end{definition}
	\subsection{Results about dual characters of flagged Weyl modules}

	A special case of dual characters of flagged Weyl modules of diagrams are Schubert polynomials:

	\begin{theorem}[\cite{KP}]
		\label{thm:kp}
		For  $w$  a permutation and  $D(w)$ its Rothe diagram we have that the Schubert polynomial $\S_w$ is 
				\[\mathfrak{S}_w = \chi_{D(w)}. \]
	\end{theorem}
 
	\begin{theorem}[cf. {\cite[Theorem 7]{FMS}}]
		For any diagram $D\subseteq [n]\times [n]$, the monomials appearing in $\chi_D$ are exactly 
		\[\left\{\prod_{j=1}^{n}\prod_{i\in C_j}x_i\ \middle|\   C\leq D \right\}.\]
	\end{theorem}

	\begin{theorem}[\cite{zeroone}]
		\label{cor:fms}
		Let $D\subseteq [n]\times [n]$ be a diagram. Fix any diagram $C^{(1)}\leq D$ and set \[\bm{m}=\prod_{j=1}^{n}\prod_{i\in C^{(1)}_j}x_i.\] 
		Let $C^{(1)}, \ldots, C^{(r)}$ be all the diagrams $C$ such that $C\leq D$ and $\prod_{j=1}^{n}\prod_{i\in C_j}x_i=\bm{m}$. Then, the coefficient of $\bm{m}$ in $\chi_D$ is equal to 
		\[[\bm{m}]\chi_D =\dim \left(\mathrm{Span}_\mathbb{C}\left\{\prod_{j=1}^{n}\det\left(Y_{D_j}^{C^{(i)}_j}\right) \ \middle|\  i\in [r] \right\}\right).\] 
		In particular, 
		$$[\bm{m}]\chi_D \leq \#\left\{ C \le D \;\middle|\; \prod_{(i, j) \in C} x_i = \mathbf{m} \right\}.$$ 
	\end{theorem}

	In light of the last inequality,  it is  natural to wonder when equality holds.  This is what Fan \& Guo \cite{FG} did:

\begin{theorem}[\cite{FG}]
\label{thm:fg}
Given a diagram $D \subseteq [n] \times [n]$, let
\[
x^D = \prod_{(i, j) \in D} x_i.
\]
Then, for a permutation $w \in S_n$,
\[
\Schub_w(x_1, \dots, x_n) = \sum_{C \le D(w)} x^C
\]
if and only if $w$ avoids the patterns $1432$ and $1423$.
\end{theorem}

In particular, Theorem \ref{thm:fg} implies:

\begin{corollary}[\cite{FG}]
\label{cor:counting}
If $w \in S_n$ avoids the patterns $1432$ and $1423$, then the coefficient of $\bm{m}$ in $\Schub_w = \chi_{D(w)}$ is equal to
\[
[\bm{m}]\Schub_w=\# \left\{ C \le D(w) \;\middle|\; \prod_{(i, j) \in C} x_i = \bm{m} \right\}.
\]
\end{corollary}

	\section{Proof of Theorems \ref{thm:main}, \ref{cor:main} and \ref{thm:cw}}
	\label{sec:main}

In this section we prove Theorems \ref{thm:main}, \ref{cor:main} and \ref{thm:cw}. We start by giving the necessary definitions and lemmas. 

\subsection{Setup for Theorems \ref{thm:main}, \ref{cor:main}  and \ref{thm:cw}.}
\begin{definition}
For words $u, v$, we write $u \le v$ if $u$ is a subword of $v$ (and $u < v$ if $u \le v$ and $u \ne v$). In other words, $u \le v$ if there is a sequence $1 \le i_1 < \cdots < i_{\left|u\right|} \le \left|v\right|$ such that $u = v(i_1) \cdots v(i_{\left|u\right|})$. The empty word $()$ is a pattern in all words.
\end{definition}

\begin{example}
Let $u = 792$ and $v = 37952$. Then $u \le v$, because $u = v(2)\,v(3)\,v(5)$.
\end{example}

\begin{definition}
For a word $v$ of length $n$, let $i_1, i_2, \dots, i_n$ be indices such that $v(i_1) < v(i_2) < \cdots < v(i_n)$. Then $\perm(v)$ is the permutation that sends $i_j \mapsto j$, that is, $\perm(v) = (i_1i_2 \cdots i_n)^{-1}$. Equivalently, $\pi = \perm(v)$ is the permutation in $S_n$ such that the relative order of $\pi_1, \dots, \pi_n$ and of $v_1, \dots, v_n$ are the same.
\end{definition}

\begin{example}
Let $v = 37952$. Note $v(5) < v(1) < v(4) < v(2) < v(3)$, so $(i_1, i_2, i_3, i_4, i_5) = (5, 1, 4, 2, 3)$. Thus $\perm(v) = (51423)^{-1} = 24531$. Notice that we can obtain $\perm(v)$ from $v$ by replacing the smallest character of $v$ with $1$, the second smallest with $2$, and so on.
\end{example}

\begin{definition}
Let $w \in S_n$ and let $v$ be a subword of $w$. We define
\[
\mathbf{x}_{w^{-1}(v)} \coloneqq (x_{w^{-1}(v(1))}, x_{w^{-1}(v(2))}, \dots, x_{w^{-1}(v(\left|v\right|))}).
\]
\end{definition}

\begin{example}
Let $w = 134265$ and $v = 3265$. Then
\[
\mathbf{x}_{w^{-1}(v)} = (x_{w^{-1}(3)}, x_{w^{-1}(2)}, x_{w^{-1}(6)}, x_{w^{-1}(5)}) = (x_2, x_4, x_5, x_6).
\]
Notice that the resulting indices will always be in ascending order.  See Figure \ref{yellow} for an illustration. \end{example}

\begin{figure}[h!]
	\centering
	\includegraphics[width=0.6\textwidth]{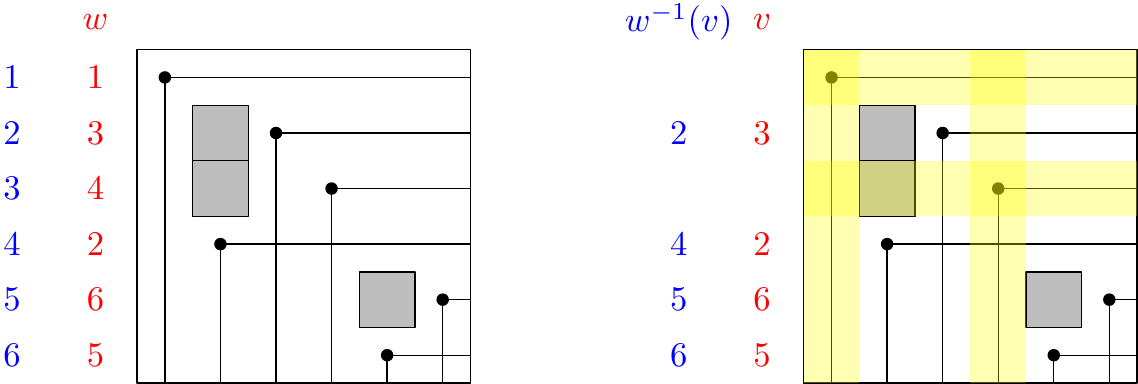}
	\caption{The left diagram is the Rothe diagram of the permutation $w = 134265$ (the permutation $w$ is noted in red to the left of the diagram). The row indices are noted in blue  to the left of the diagram. The right diagram shows the subword $v = 3265$ of   $w = 134265$ graphically: it is obtained by removing the yellow highlighted rows and columns from the Rothe diagram of $w$. The indices $w^{-1}(v)$ shown in blue to the left of the diagram are simply the row indices corresponding to this graphical presentation of the  subword $v = 3265$ of   $w = 134265$.}
	\label{yellow}
\end{figure}

\begin{definition}
Given a diagram $D \subseteq [n] \times [n]$ and sets of indices $K, L \subseteq [n]$ with $\#K = \#L$, let $\widehat{D}_{K, L}$ denote the diagram obtained from $D$ by keeping only the boxes in rows $K$ and columns $L$:
\[
\widehat{D}_{K, L} = \{ (i, j) \in D \mid i \in K, j \in L \}.
\]
\end{definition}

\begin{definition}
Suppose $C \le D(w)$ for some permutation $w \in S_n$. Then, for any subword $v \le w$, we define $\widehat{C}_v$ to be the diagram obtained by keeping only the boxes in the rows corresponding to $v$. That is, $\widehat{C}_v \coloneqq \widehat{C}_{K, L}$, where $L = \{v(1), v(2), \dots, v(\left|v\right|)\}$ and $K = \{w^{-1}(v(1)), w^{-1}(v(2)), \dots, w^{-1}(v(\left|v\right|))\}$.
\end{definition}

\subsection{Theorem \ref{thm:main} and its proof.}

\begin{theorem1*} 
\textit{Let $w \in S_n$ be a $1432$ and $1423$ avoiding permutation and let $u$ be a subword of $w$. Then}
\begin{equation} \label{eqn:f_w}
\sum_{u \le v \le w} (-1)^{\left|w\right|-\left|v\right|} M_{w, v} \Schub_{\perm(v)}(\mathbf{x}_{w^{-1}(v)}) \in \Z_{\ge0}[x_1, \dots, x_n],
\end{equation}
{\it where}
\begin{equation*}
M_{w, v} \coloneqq \prod_{(i, j) \in D(w) \setminus \widehat{D(w)}_v} x_i.
\end{equation*}
\end{theorem1*}

\begin{example} \label{ex1}
Let $w = 2143$ and $u = 43$. Then
\[
\{ v \mid u \le v \le w \} = \{ 2143, 143, 243, 43 \},
\]
so the alternating sum in (\ref{eqn:f_w}) becomes 
\begin{align} \label{?1}
&M_{w, 2143}\Schub_{2143}(x_1, x_2, x_3, x_4) - M_{w, 143}\Schub_{132}(x_2, x_3, x_4) - M_{w, 243}\Schub_{132}(x_1, x_3, x_4) + M_{w, 43}\Schub_{21}(x_3, x_4) \\ \nonumber
&\quad = 1 \cdot (x_1^2 + x_1x_2 + x_1x_3) - x_1 \cdot (x_2 + x_3) - x_1 \cdot (x_1 + x_3) + x_1 \cdot (x_3) \\ \nonumber
&\quad = 0,
\end{align}
which indeed has nonnegative coefficients. See Figure \ref{fig1} for an illustration.
\end{example}

\begin{figure}[h!]
	\centering
	\includegraphics[width=\textwidth]{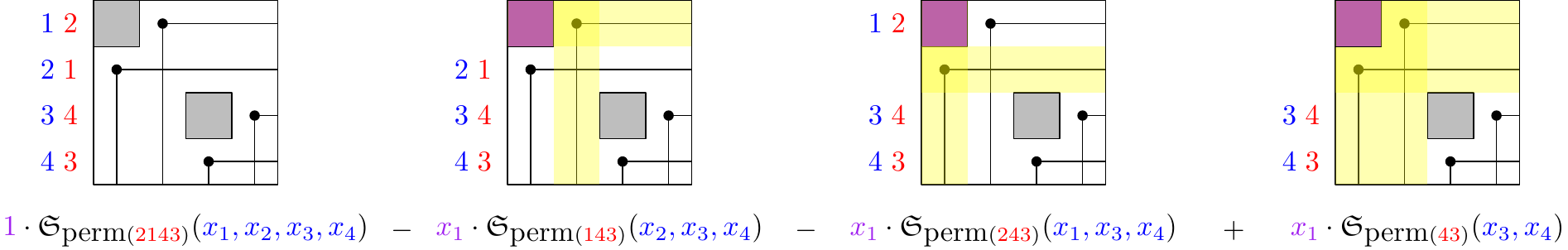}
	\caption{The four diagrams in this figure correspond left to right to the subwords $\{ v \mid u \le v \le w \} = \{ 2143, 143, 243, 43 \}$ for  $w = 2143$ and $u = 43$ as in Example \ref{ex1}. 
 These in turn yield the Schubert polynomials in the expression  \eqref{?1}. The red numbers on the left of the diagrams signify these subwords; the blue numbers are the row numbers yielding the variables of the corresponding Schubert polynomials in the expression  \eqref{?1}. The purple boxes correspond to the boxes of the Rothe diagram of $w = 2143$ that are removed in order to obtain $v$; graphically these are the boxes struck by yellow if the yellow highlighted rows and columns are extended; the row indices of these boxes yield the monomials $M_{w,v}$.}
	 \label{fig1}
\end{figure}

\begin{example} \label{ex2}
Let $w = 1342$ and $u = 42$. Then
\[
\{ v \mid u \le v \le w \} = \{ 1342, 142, 342, 42 \},
\]
so the alternating sum in (\ref{eqn:f_w}) becomes 
\begin{align} \label{?2}
&M_{w, 1342}\Schub_{1342}(x_1,x_2, x_3, x_4) - M_{w, 142}\Schub_{132}(x_1, x_3, x_4) - M_{w, 342}\Schub_{231}(x_2, x_3, x_4) + M_{w, 42}\Schub_{21}(x_3, x_4) \\ \nonumber
&\quad = 1 \cdot (x_1x_2 + x_1x_3 + x_2x_3) - x_2 \cdot (x_1 + x_3) - 1 \cdot (x_2 x_3) + x_2 \cdot (x_3) \\ \nonumber
&\quad = x_1x_3,
\end{align}
which indeed has nonnegative coefficients. See Figure \ref{fig2} for an illustration.
\end{example}

\begin{figure}[h!]
	\centering
	\includegraphics[width=\textwidth]{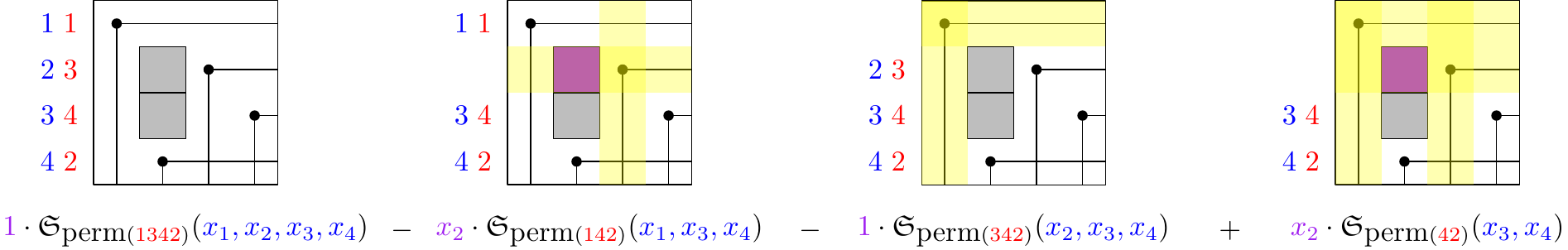}
	\caption{The four diagrams in this figure correspond left to right to the subwords $\{ v \mid u \le v \le w \} = \{ 1342, 142, 342, 42 \}$ for  $w = 1342$ and $u = 42$ as in Example \ref{ex2}. 
 These in turn yield the Schubert polynomials in the expression  \eqref{?2}. The red numbers on the left of the diagrams signify these subwords; the blue numbers are the row numbers yielding the variables of the corresponding Schubert polynomials in the expression  \eqref{?2}. The purple boxes correspond to the boxes of the Rothe diagram of $w = 1342$ that are removed in order to obtain $v$; graphically these are the boxes struck by yellow if the yellow highlighted rows and columns are extended; the row indices of these boxes yield the monomials $M_{w,v}$.}
\label{fig2}
\end{figure}

To aid the proof of Theorem \ref{thm:main} we extend Corollary \ref{cor:counting} to words:

\begin{lemma}
\label{lem:counting}
Let $w \in S_n$ be a $1432$ and $1423$ avoiding permutation, and let $v$ be a subword of $w$. Then the coefficient of $\mathbf{m}$ in $\Schub_{\perm(v)}(\mathbf{x}_{w^{-1}(v)})$ is equal to
\[
\# \left\{ C \le \widehat{D(w)}_v \;\middle|\; \prod_{(i, j) \in C} x_i = \mathbf{m} \text{ and $C$'s boxes all lie in rows $K$} \right\},
\]
where $K = \{ w^{-1}(v(1)), w^{-1}(v(2)), \dots, w^{-1}(v(\left|v\right|)) \}$.
\end{lemma}
\begin{proof}
Fix $\mathbf{m}$, and let
\[
A = \left\{ C \le \widehat{D(w)}_v \;\middle|\; \prod_{(i, j) \in C} x_i = \mathbf{m} \text{ and $C$'s boxes all lie in rows $K$} \right\}.
\]
If $\mathbf{m}$ is divisible by some $x_i$ where $i \notin K$, then the coefficient of $\mathbf{m}$ in $\Schub_{\perm(v)}(\mathbf{x}_{w^{-1}(v)})$ is $0$, and no diagram $C$ with boxes only in rows $K$ can ever satisfy $\prod_{(i, j) \in C} x_i = \mathbf{m}$, so $\left|A\right| = 0$ and we are done.

Let $k_1 < k_2 < \cdots < k_{\left|v\right|}$ be the elements of $K$. By the previous discussion, we may as well assume that we can write
\[
\mathbf{m} = \prod_{i = 1}^{\left|v\right|} x_{k_i}^{\alpha_i}
\]
for some nonnegative integers $\alpha_i$. Define
\[
\mathbf{m}' = \prod_{i = 1}^{\left|v\right|} x_i^{\alpha_i},
\]
which is simply $\mathbf{m}$ under the reindexing $x_{k_i} \mapsto x_i$. Since $w$ is 1432 and 1423 avoiding, so are $v$ and $\perm(v)$, thus by Corollary \ref{cor:counting}, the coefficient of $\mathbf{m'}$ in $\Schub_{\perm(v)}(x_1, \dots, x_{\left|v\right|})$ is equal to $\left|B\right|$, where
\[
B = \left\{ C \le D(\perm(v)) \;\middle|\; \prod_{(i, j) \in C} x_i = \mathbf{m'} \right\}.
\]

Consider the function $f \colon B \to A$ given by
\[
f(C) = \{ (k_i, v(j)) \mid (i, j) \in C \}.
\]
Notice that the boxes of $f(C)$ all lie in rows $K$, and since $\prod_{(i, j) \in C} x_i = \mathbf{m'}$, we have $\prod_{(i, j) \in f(C)} x_i = \mathbf{m}$. Furthermore, from the definition of Rothe diagrams, if $(i, j) \in D(\perm(v))$ then $(k_i, v(j)) \in D(w)$, so $f(D(\perm(v))) \le \widehat{D(w)}_v$. For $C, C' \in B$, observe that if $C \le C'$ then $f(C) \le f(C')$, so it follows that $f(C) \le \widehat{D(w)}_v$ for all $C \in B$ and $f$ is well-defined.

$f$ is clearly injective by construction. To see that it is surjective, note that if $C \in A$, then $C \le \widehat{D(w)}_v$, so every box in $C$ is of the form $(k_i, v(j))$ and the diagram
\[
C' = \{ (i, j) \mid (k_i, v(j)) \in C \}
\]
is easily seen to be a member of $A$, with $f(C') = C$. Therefore,
\begin{align*}
[\mathbf{m}] \Schub_{\perm(v)}(\mathbf{x}_{w^{-1}(v)}) &= [\mathbf{m}] \Schub_{\perm(v)}(x_{k_1}, x_{k_2}, \dots, x_{k_{\left|v\right|}}) \\
&= [\mathbf{m'}] \Schub_{\perm(v)}(x_1, x_2, \dots, x_{\left|v\right|}) \\
&= \left|B\right| \\
&= \left|A\right|.
\end{align*}
\end{proof}

\begin{proof}[Proof of Theorem \ref{thm:main}]
We must show that for every monomial $\mathbf{m}$,
\[
[\mathbf{m}]\sum_{u \le v \le w} (-1)^{\left|w\right|-\left|v\right|} M_{w, v} \Schub_{\perm(v)}(\mathbf{x}_{w^{-1}(v)}) \ge 0;
\]
equivalently,
\begin{equation}
\sum_{u \le v \le w} (-1)^{\left|w\right|-\left|v\right|} [\mathbf{m}] M_{w, v} \Schub_{\perm(v)}(\mathbf{x}_{w^{-1}(v)}) \ge 0.
\end{equation}

Fix $\mathbf{m}$. For any subword $v \le w$, let $K_v \coloneqq \{w^{-1}(v(1)), w^{-1}(v(2)), \dots, w^{-1}(v(\left|v\right|))\}$ (the `rows corresponding to $v$'). Using Lemma \ref{lem:counting}, we find that
\begin{align*}
[\mathbf{m}] M_{w, v} \Schub_{\perm(v)}(\mathbf{x}_{w^{-1}(v)}) &= \left[\frac{\mathbf{m}}{M_{w, v}}\right] \Schub_{\perm(v)}(\mathbf{x}_{w^{-1}(v)}) \\
&= \# \left\{ C \le \widehat{D(w)}_v \;\middle|\; \prod_{(i, j) \in C} x_i = \frac{\mathbf{m}}{M_{w, v}} \text{ and $C$'s boxes all lie in rows $K_v$} \right\}.
\end{align*}

Consider the two families of sets
\begin{align*}
A_v &\coloneqq \left\{ C \le \widehat{D(w)}_v \;\middle|\; \prod_{(i, j) \in C} x_i = \frac{\mathbf{m}}{M_{w, v}} \text{ and $C$'s boxes all lie in rows $K_v$} \right\}, \\
B_v &\coloneqq \left\{ C \le D(w) \;\middle|\; \prod_{(i, j) \in C} x_i = \mathbf{m}, C \setminus \widehat{C}_v = D(w) \setminus \widehat{D(w)}_v \text{ and $C$'s boxes all lie in rows $K_w$} \right\}.
\end{align*}

Since $v \le w$, $K_v \subseteq K_w$, and also the boxes of $D(w) \setminus \widehat{D(w)}_v$ all lie in rows $K_w \setminus K_v$. Thus, $D(w) \setminus \widehat{D(w)}_v$ is disjoint from every $C \in A_v$, and so there is an obvious injection $f$ from $A_v$ to $B_v$ defined by
\[
f(C) \coloneqq C \sqcup (D(w) \setminus \widehat{D(w)}_v).
\]
We claim $f$ is surjective. Indeed, given $C \in B_v$, the diagram $C' = C \setminus (D(w) \setminus \widehat{D(w)}_v)$ is easily seen to be a member of $A_v$, and of course $f(C') = C$.

Therefore,
\begin{equation}
[\mathbf{m}] M_{w, v} \Schub_{\perm(v)}(\mathbf{x}_{w^{-1}(v)}) = \left|A_v\right| = \left|B_v\right|,
\end{equation}
and so it suffices to show that
\begin{equation} \label{eqn:wts}
\sum_{u \le v \le w} (-1)^{\left|w\right|-\left|v\right|} \left|B_v\right| \ge 0.
\end{equation}

Notice that, if $u \le v \le v' \le w$, then $B_v \subseteq B_{v'}$, and for all $u \le v, v' \le w$, $B_u \cap B_v = B_{u \wedge v}$, where $u \wedge v$ denotes the maximal word contained in both $u$ and $v$. Let $I = \{ v \mid u \le v \le w \text{ and } \left|v\right| = \left|w\right| - 1 \}$. Then, using inclusion-exclusion, we find that
\begin{align}
\sum_{u \le v \le w} (-1)^{\left|w\right|-\left|v\right|} \left| B_v \right| &= \left|B_w\right| - \sum_{v_1 \in I} \left| B_{v_1} \right| + \sum_{v_1, v_2 \in I} \left| B_{v_1} \cap B_{v_2} \right| - \cdots \\
&= \left|B_w\right| - \left| \bigcup_{v \in I} B_v \right| \\
&= \left| B_w \setminus \bigcup_{v \in I} B_v \right|.
\end{align}
This quantity is necessarily non-negative, as desired.
\end{proof}

By setting all $x_i$'s to $1$ in  Theorem \ref{thm:main} we obtain:

\begin{theorem3*}  
\textit{Let $w \in S_n$ be a $1432$ and $1423$ avoiding permutation. If $u$ is a subword of $w$, then}
\[
\sum_{u \le v \le w} (-1)^{\left|w\right|-\left|v\right|} \Schub_{\textup{perm}(v)}({\bf 1})\geq 0.
\]
\end{theorem3*}

We conjecture (Conjecture \ref{conj:1.3}) that  Theorem \ref{cor:main}  generalizes to all permutations $w \in S_n$. 

\subsection{Gao's conjecture \ref{conj:gao}, Theorem \ref{thm:cw} and its proof.}
\label{sec:gao}
	Gao \cite{Gao}
defined a sequence of integers $\{c_u\}_{m \ge 1, u \in S_m}$ recursively, as follows:
\begin{equation} \label{eqn:c_w_defn}
c_w \coloneqq \Schub_w(\mathbf{1}) - 1 - \sum_{\left|u\right| < \left|w\right|} c_u p_u(w),
\end{equation}
where $\left|u\right| = m$ if $u \in S_m$, and $p_u(w)$ is the number of occurrences of $u$ as a pattern in $w$.

Gao
showed that $c_w = 0$ whenever $w(n) = n$, so the definition of $c_w$ can be extended to all $w \in S_\infty$. In the same paper, he conjectured the following:
\begin{conjecture}(\cite[Conjecture 3.2]{Gao}) \label{conj:gao2}
We have $c_w \ge 0$ for all $w \in S_\infty$.
\end{conjecture}

Notice that $p_u(w) = \# \{\text{words $v$ such that $u = \perm(v)$ and $v \le w$}\}$. Thus, we can rewrite (\ref{eqn:c_w_defn}) as
\begin{equation} \label{eqn:c_w_defn2}
c_w = \Schub_w({\bf 1}) - \sum_{v < w} c_v,
\end{equation}
where the $-1$ has been absorbed into the sum as $c_{()}$. Note that this perspective explains the equivalence of  Conjectures \ref{conj:gao} and \ref{conj:gao2}.

By inclusion-exclusion, (\ref{eqn:c_w_defn2}) is equivalent to
\begin{equation} \label{eqn:c_w_defn3}
c_w = \sum_{v \le w} (-1)^{\left|w\right|-\left|v\right|} \Schub_v({\bf 1}).
\end{equation}

Thus, Theorem \ref{cor:main} immediately implies:

\begin{theorem} \label{thm:conj} Conjecture \ref{conj:gao2} (equivalently, Conjecture \ref{conj:gao})  holds for $1432$ and $1423$ avoiding permutations  $w \in S_{\infty}$. \end{theorem} 

	Moreover, Theorem \ref{thm:cw} below provides a combinatorial interpretation for $c_w$ when $w$ is $1432$ and $1423$ avoiding. 

\begin{definition}
	Given diagrams $C,D\subseteq [n]\times[n]$ and $k,l\in [n]$, let $\hc$ and  $\hd$ denote the diagrams obtained from $C$ and $D$ by removing any boxes in row $k$ or column $l$. 
	Fix a diagram $D$. For each diagram $\hc$, let its  {\bf augmentation} with respect to the diagram $D$ be:
	\[{\hc}_{\rm aug}=\hc \cup \{(k,i) \mid (k,i)\in D\} \cup \{(i,l) \mid (i,l) \in D \}\subseteq [n]\times [n].\]  
\end{definition}

By tracing the proof of Theorem \ref{thm:main} for the case $u = ()$, we can obtain an interpretation of the coefficient of $\mathbf{m}$ in
$\sum_{v \le w} (-1)^{\left|w\right| - \left|v\right|} M_{w, v}\Schub_v(\mathbf{x}_{w^{-1}(v)})$ in terms of augmentations of diagrams $\widehat{C} \le \widehat{D(w)}$. In particular, we readily obtain:

	\begin{theorem2*}  \textit{For $1432$ and $1423$ avoiding permutations  $w \in S_{\infty}$ the value of $c_w$ is the number of diagrams $C \le D(w)$ that cannot be written as $\widehat{C}_\textup{aug}$ for some} $\widehat{C} \le \widehat{D(w)}$. 
	\end{theorem2*}

	We conclude this section by illustrating Theorem \ref{thm:cw}  for permutations $1342$ and $12453$. 

\begin{example} Computation yields $c_{1342}=0$. We have that $D(1342)=\{(2,2), (3,2)\}$ and thus the diagrams $C^1=\{(2,2), (3,2)\}$, $C^2=\{(1,2), (3,2)\}$, $C^3=\{(1,2), (2,2)\}$ are all the diagrams $C\leq D(1342)$.  Note that $C^1=\widehat{C^1}_\textup{aug}$ with respect to $D(1342)$ with $k=3, l=4$ (or with $k=2, l=3$); $C^2=\widehat{C^2}_\textup{aug}$ with respect to $D(1342)$ with $k=3, l=4$; $C^3=\widehat{C^3}_\textup{aug}$ with respect to $D(1342)$ with $k=2, l=3$. Thus, the number of diagrams $C \le D(1342)$ that cannot be written as $\widehat{C}_\textup{aug}$ for some $\widehat{C} \le \widehat{D(1342)}$ is $0$ yielding $c_{1342}=0$.
\end{example}

\begin{example} Computation yields $c_{12453}=1$. We have that $D(12453)=\{(3,3), (4,3)\}$ and thus the diagrams $C^1 = \{(3,3), (4,3)\}$, $C^2 = \{(2,3), (4,3)\}$, $C^3 = \{(1,3), (4,3)\}$, $C^4 = \{(2,3), (3,3)\}$, $C^5 = \{(1,3), (3,3)\}$, $C^6 = \{(1,3), (2,3)\}$ are all the diagrams $C\leq D(12453)$.   Note that $C^1=\widehat{C^1}_\textup{aug}$ with respect to $D(12453)$ with $k=4, l=5$ (or with $k=3, l=4$); $C^2=\widehat{C^2}_\textup{aug}$ with respect to  $D(12453)$ with $k=4, l=5$; $C^3=\widehat{C^3}_\textup{aug}$ with respect to  $D(12453)$ with $k=4, l=5$; $C^4=\widehat{C^4}_\textup{aug}$ with respect to  $D(12453)$ with $k=3, l=4$;  $C^5=\widehat{C^5}_\textup{aug}$ with respect to  $D(12453)$ with $k=3, l=4$.  Note also that $C^6$ cannot be written as $\widehat{C^6}_\textup{aug}$ for some $\widehat{C^6} \le \widehat{D(12453)}$.  Thus, the number of diagrams $C \le D(12453)$ that cannot be written as $\widehat{C}_\textup{aug}$ for some $\widehat{C} \le \widehat{D(1342)}$ is $1$ yielding $c_{12453}=1$.

\end{example}

	\section{Proof of Theorem \ref{thm:gen}}
	\label{sec:gen}
	The main result of this section is a generalization of Theorems \ref{thm:fms} and \ref{thm:gen}:		 

\begin{theorem}\label{thm:gen1} 
Fix a diagram $D \subseteq [n] \times [n]$ and let $\widehat{D}$ be the diagram obtained from $D$ by removing any boxes in row $k$ or column $l$. If there is some diagram $K \in \Purple_{k, l}(D)$ such that
\[
M(x_1, \dots, x_n) = \prod_{(i, j) \in K} x_i,
\]
then  \[
\chi_D(x_1, \dots, x_n) - M(x_1, \dots, x_n) \chi_{\widehat{D}}(x_1, \dots, x_{k-1}, 0, x_{k+1}, \dots, x_n)  \in \Z_{\ge0}[x_1, \dots, x_n].\]
\end{theorem}

Theorem \ref{thm:gen} is a special case of Theorem \ref{thm:gen1} when $D$ is a Rothe diagram of a permutation.

We now proceed to define the set of diagrams $\Purple_{k, l}(D)$ used in the statement of Theorem \ref{thm:gen1} above.

\begin{definition}
Fix a diagram $D \subseteq [n]\times [n]$ and integers $k, l \in [n]$. Define $\purple_{k, l}(D)$ to be the set of boxes $(i, j)$ such that:  
\begin{itemize}
\item there is some $C \le D$ such that $(i, j) \in C$, but
\item there is no $C \le D$ such that $\widehat{C}_{k, l} \le \widehat{D}_{k, l}$ and $(i, j) \in \widehat{C}_{k, l}$.
\end{itemize}
\end{definition}

\begin{definition}
Fix a diagram $D \subseteq [n]\times [n]$ and integers $k, l \in [n]$. Define $\Purple_{k, l}(D)$ to be the smallest set satisfying the following:
\begin{itemize}
\item $D \setminus \widehat{D}_{k, l} \in \Purple_{k, l}(D)$, and
\item if $K \in \Purple_{k, l}(D)$, $K' \le K$ and $K' \subseteq \purple_{k, l}(D)$, then $K' \in \Purple_{k, l}(D)$.
\end{itemize}
\end{definition}

\begin{example} \label{p1}
Let $D = D(15243)$, $k = 5$ and $l = 3$. Then $\purple_{k, l}(D) = \{(1, 3), (2, 3), (3, 3), (4, 3)\}$ and $\Purple_{k, l}(D) = \{\{(2, 3), (4, 3)\}, \{(1, 3), (4, 3)\}, \{(2, 3), (3, 3)\}, \{(1, 3), (3, 3)\}, \{(1, 3), (2, 3)\}\}$. As a result, the set of monomials $M$ produced by Theorem \ref{thm:gen} is $\{x_2x_4, x_1x_4, x_2x_3, x_1x_3, x_1x_2\}$. In this case, these are \textbf{all} the monomials $M$ for which
\[
\chi_D(x_1,x_2,x_3,x_4,x_5) - M(x_1,x_2,x_3,x_4,x_5) \chi_{\widehat{D}}(x_1,x_2,x_3,x_4,0) \in \Z_{\ge0}[x_1,x_2,x_3,x_4,x_5].
\] See Figure \ref{fig:p1} for an illustration.
\end{example}

\begin{figure}[h!]
	\centering
	\includegraphics[width=0.85\textwidth]{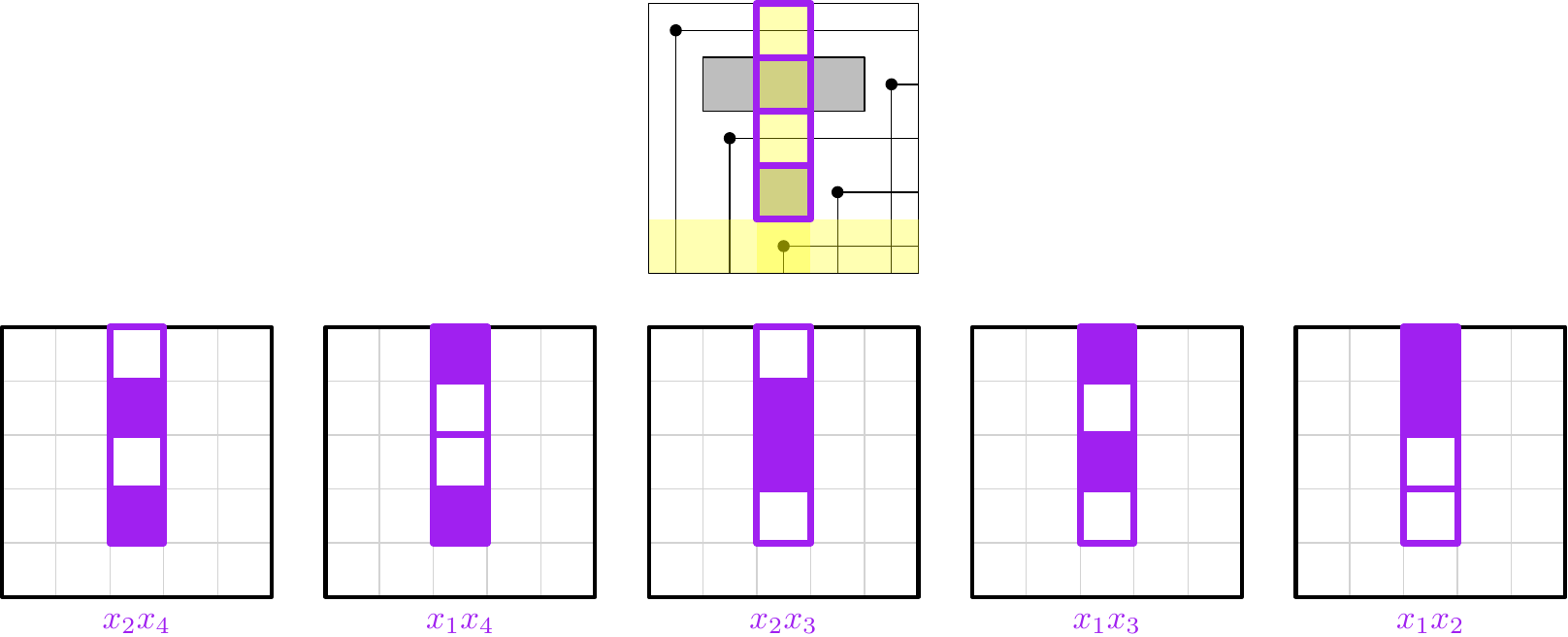}
	\caption{The diagram in the first row shows the Rothe diagram of the permutation $15243$. The yellow highlighted row and column correspond to removing row indexed $k=5$ and column indexed $l=3$. The boxes with purple boundary are $\purple_{k, l}(D) = \{(1, 3), (2, 3), (3, 3), (4, 3)\}$. The second row of the figure shows $\Purple_{k, l}(D) = \{\{(2, 3), (4, 3)\}, \{(1, 3), (4, 3)\}, \{(2, 3), (3, 3)\}, \{(1, 3), (3, 3)\}, \{(1, 3), (2, 3)\}\}$ along with the corresponding monomials below each diagram.}
	\label{fig:p1}
\end{figure}

\begin{example} \label{p2}
Let $D = D(15243)$ and $k = l = 4$. Then $\purple_{k, l}(D) = \{(1, 4), (2, 4), (3, 3), (4, 3)\}$ and $\Purple_{k, l}(D) = \{\{(2, 4), (4, 3)\}, \{(1, 4), (4, 3)\}, \{(2, 4), (3, 3)\}, \{(1, 4), (3, 3)\}\}$. As a result, the set of monomials $M$ produced by Theorem \ref{thm:gen} is $\{x_2x_4, x_1x_4, x_2x_3, x_1x_3\}$. However,
\[
\chi_D(x_1,x_2,x_3,x_4,x_5) - (x_1x_2) \chi_{\widehat{D}}(x_1,x_2,x_3,0,x_5) \in \Z_{\ge0}[x_1,x_2,x_3,x_4,x_5],
\]
so in this case the monomials prescribed by Theorem \ref{thm:gen} are not the only monomials that could work. See Figure \ref{fig:p2} for an illustration.
\end{example}

\begin{figure}[h!]
	\centering
	\includegraphics[width=0.67\textwidth]{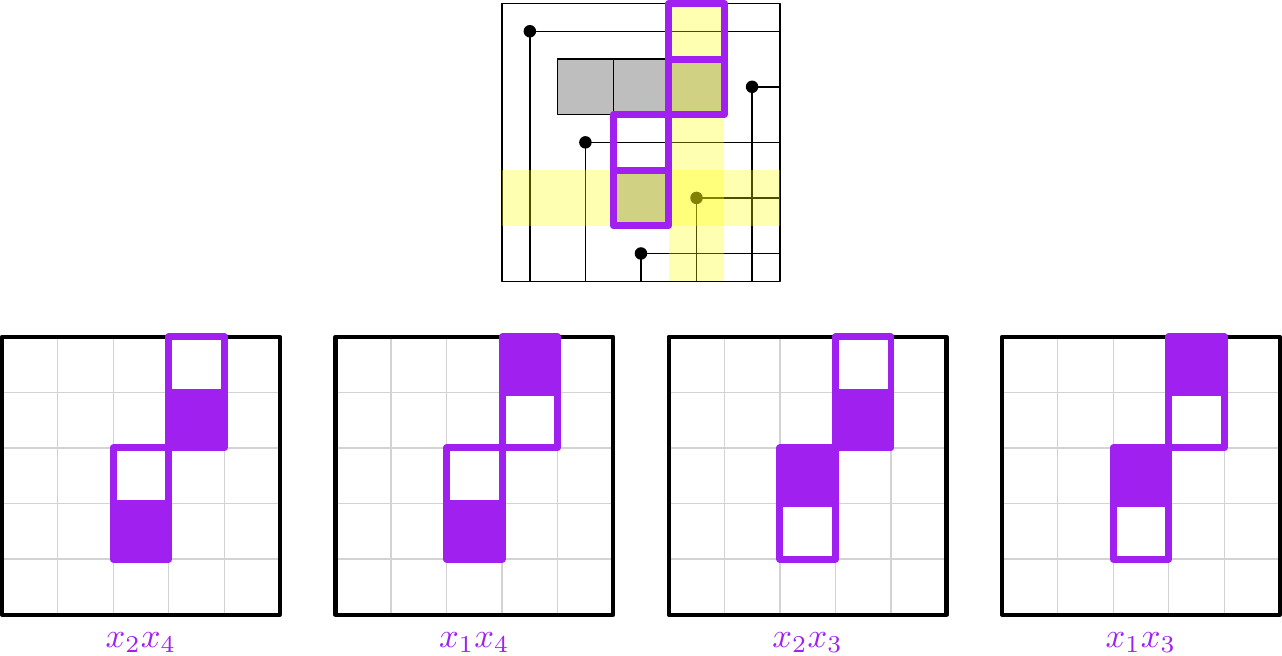}
	\caption{The diagram in the first row shows the Rothe diagram of the permutation $15243$. The yellow highlighted row and column correspond to removing row indexed $k=4$ and column indexed $l=4$. The boxes with purple boundary are $\purple_{k, l}(D) = \{(1, 4), (2, 4), (3, 3), (4, 3)\}$. The second row of the figure shows $\Purple_{k, l}(D) = \{\{(2, 4), (4, 3)\}, \{(1, 4), (4, 3)\}, \{(2, 4), (3, 3)\}, \{(1, 4), (3, 3)\}\}$ along with the corresponding monomials below each diagram.}
	\label{fig:p2}
\end{figure}

The following lemma follows immediately from the definitions:
\begin{lemma} \label{lem:aug}
Let $C, D \subseteq [n] \times [n]$ be diagrams, $k, l \in [n]$, and $K \in \Purple_{k, l}(D)$. If $\widehat{C}_{k, l} \le \widehat{D}_{k, l}$, then $\widehat{C}_{k, l} \cup K \le D$ and $\widehat{C}_{k, l} \cap K = \varnothing$.
\end{lemma}

The following lemma generalizes \cite[Lemma 5.7]{zeroone}:
\begin{lemma} \label{lem:lift}
Fix a diagram $D$, integers $k, l \in [n]$, $K \in \Purple_{k, l}(D)$, and let $\widehat{D}$ denote $\widehat{D}_{k, l}$. Let $\{\widehat{C}^{(i)}\}_{i \in [m]}$ be a set of diagrams with $\widehat{C}^{(i)} \le \widehat{D}$ for each $i$, and denote $\widehat{C}^{(i)} \cup K$ by $C^{(i)}$ for $i \in [m]$. If the polynomials $\left\{ \displaystyle\prod_{j \in [n]} \det(Y^{C^{(i)}_j}_{D_j}) \right\}_{i \in [m]}$ are linearly dependent, then so are the polynomials $\left\{ \displaystyle\prod_{j \in [n] \setminus \{l\}} \det(Y^{\widehat{C}^{(i)}_j}_{\widehat{D}_j}) \right\}_{i \in [m]}$.
\end{lemma}
\begin{proof}
We are given that
\begin{equation} \label{lem:lift-eqn:1}
\sum_{i \in [m]} c_i \prod_{j \in [n]} \det(Y^{C^{(i)}_j}_{D_j}) = 0
\end{equation}
for some constants $(c_i)_{i \in [m]} \in \C^m$ not all zero. Since $C^{(i)} = \widehat{C}^{(i)} \cup K$ for $\widehat{C}^{(i)} \le \widehat{D}$ we have that $C^{(i)}_l = K_l$ for every $i \in [m]$. Thus, (\ref{lem:lift-eqn:1}) can be rewritten as
\begin{equation}
\det(Y^{K_l}_{D_l}) \left( \sum_{i \in [m]} c_i \prod_{j \in [n] \setminus \{l\}} \det(Y^{C^{(i)}_j}_{D_j}) \right) = 0.
\end{equation}

However, since $\det(Y^{K_l}_{D_l}) \ne 0$, we conclude that
\begin{equation} \label{lem:lift-eqn:3}
\sum_{i \in [m]} c_i \prod_{j \in [n] \setminus \{l\}} \det(Y^{C^{(i)}_j}_{D_j}) = 0.
\end{equation}

First consider the case that the only boxes of $D$ in row $k$ or column $l$ are those in $D_l$. If this is the case then
\begin{equation} 
\prod_{j \in [n] \setminus \{l\}} \det(Y^{\widehat{C}^{(i)}_j}_{\widehat{D}_j}) = \prod_{j \in [n] \setminus \{l\}} \det(Y^{C^{(i)}_j}_{D_j})
\end{equation}
for each $i \in [m]$. Therefore,
\begin{equation} \label{lem:lift-eqn:5}
\sum_{i \in [m]} c_i \prod_{j \in [n] \setminus \{l\}} \det(Y^{\widehat{C}^{(i)}_j}_{\widehat{D}_j}) = \sum_{i \in [m]} c_i \prod_{j \in [n] \setminus \{l\}} \det(Y^{C^{(i)}_j}_{D_j}).
\end{equation}

Combining (\ref{lem:lift-eqn:3}) and (\ref{lem:lift-eqn:5}) we obtain that the polynomials $\left\{ \prod_{j \in [n] \setminus \{l\}} \det(Y^{\widehat{C}^{(i)}_j}_{\widehat{D}_j}) \right\}_{i \in [m]}$ are linearly dependent, as desired.

Now, suppose that there are boxes of $D$ in row $k$ that are not in $D_l$. Let $j_1 < \dots < j_p$ be all indices $j \ne l$ such that $D_j = \widehat{D}_j \cup \{k\}$. Then, for each $i \in [m]$ and $q \in [p]$, $C^{(i)}_{j_q} \setminus \widehat{C}^{(i)}_{j_q} = K_{j_q}$. For each $q \in [p]$, let $k_q$ be the only element of $K_{j_q}$; then (\ref{lem:lift-eqn:3}) implies that
\begin{equation} \label{lem:lift-eqn:6}
\left[ \prod_{q \in [p]} y_{k_q k} \right] \sum_{i \in [m]} c_i \prod_{j \in [n] \setminus \{l\}} \det(Y^{C^{(i)}_j}_{D_j}) = 0.
\end{equation}

However,
\begin{equation}
\left[ \prod_{q \in [p]} y_{k_q k} \right] \prod_{j \in [n] \setminus \{l\}} \det(Y^{C^{(i)}_j}_{D_j}) = \prod_{j \in [n] \setminus \{l\}} \det(Y^{\widehat{C}^{(i)}_j}_{\widehat{D}_j}),
\end{equation}
as is seen by Laplace expansion on the $k_q$th row of $\det(Y^{C^{(i)}_{j_q}}_{D_{j_q}})$, and therefore
\begin{equation} \label{lem:lift-eqn:8}
\left[ \prod_{q \in [p]} y_{k_q k} \right] \sum_{i \in [m]} c_i \prod_{j \in [n] \setminus \{l\}} \det(Y^{C^{(i)}_j}_{D_j}) = \sum_{i \in [m]} c_i \prod_{j \in [n] \setminus \{l\}} \det(Y^{\widehat{C}^{(i)}_j}_{\widehat{D}_j}).
\end{equation}

Thus, (\ref{lem:lift-eqn:6}) and (\ref{lem:lift-eqn:8}) imply that
\begin{equation}
\sum_{i \in [m]} c_i \prod_{j \in [n] \setminus \{l\}} \det(Y^{\widehat{C}^{(i)}_j}_{\widehat{D}_j}) = 0,
\end{equation}
as desired.
\end{proof}

We are now ready to prove Theorem \ref{thm:gen1}:
\bigskip

\noindent{\it Proof of Theorem \ref{thm:gen1}} Let $M = M(x_1, \dots, x_n)$.
Suppose there is some $K \in \Purple_{k, l}(D)$ such that
\[
M(x_1, \dots, x_n) = \prod_{(i, j) \in K} x_i.
\]

We must show that $[M\mathbf{m}] \chi_D \ge [\mathbf{m}] \chi_{\widehat{D}}$ for each monomial $\mathbf{m}$ of $\chi_{\widehat{D}}$ not divisible by $x_k$. Let $\widehat{\mathcal{C}}$ be the set of diagrams $\widehat{C}$ such that $\widehat{C} \le \widehat{D}$ and $\prod_{(i, j) \in \widehat{C}} x_i = \mathbf{m}$. By Corollary 5.5,
\[
[\mathbf{m}]\chi_{\widehat{D}} = \dim \left( \textstyle \Span_\C \displaystyle \left\{ \prod_{j=1}^n \det(Y^{\widehat{C}_j}_{\widehat{D}_j}) \;\middle|\; \widehat{C} \in \widehat{\mathcal{C}} \right\} \right).
\]

Let $\mathcal{C} = \{ \widehat{C} \cup K \mid \widehat{C} \in \widehat{\mathcal{C}} \}$. By Lemma \ref{lem:aug}, every $C \in \mathcal{C}$ satisfies $C \le D$ and $\prod_{(i, j) \in C} x_i = M\mathbf{m}$, so Corollary 5.5 implies that
\[
[M\mathbf{m}]\chi_D \ge \dim \left( \textstyle \Span_\C \displaystyle \left\{ \prod_{j=1}^n \det(Y^{C^{(i)}_j}_{D_j}) \;\middle|\; C \in \mathcal{C} \right\} \right).
\]
Note the inequality, which is because we have only a subset of the $C$.

Finally, Lemma \ref{lem:lift} implies that
\[
\dim \left( \textstyle \Span_\C \displaystyle \left\{ \prod_{j=1}^n \det(Y^{C^{(i)}_j}_{D_j}) \;\middle|\; C \in \mathcal{C} \right\} \right) \ge \dim \left( \textstyle \Span_\C \displaystyle \left\{ \prod_{j=1}^n \det(Y^{\widehat{C}^{(i)}_j}_{\widehat{D}_j}) \;\middle|\; \widehat{C} \in \widehat{\mathcal{C}} \right\} \right),
\]
so $[M\mathbf{m}]\chi_D \ge [\mathbf{m}]\chi_{\widehat{D}}$ for each monomial $\mathbf{m}$ of $\chi_{\widehat{D}}$ not divisible by $x_k$.

\qed

\section{Problems and Conjectures}
\label{sec:conj}

The results of this paper naturally give rise to the following Conjectures and Problems. 

\subsection{Extending Theorem \ref{cor:main}}	\begin{conjecture} \label{conj:1.3}
		Let $w \in S_n$. If $u$ is a subword of $w$, then
$$
			\sum_{u \le v \le w} (-1)^{\left|w\right|-\left|v\right|} \Schub_{\textup{perm}(v)}({\bf 1})\geq 0.
$$
	 		\end{conjecture}

Theorem \ref{cor:main} confirms the above conjecture for $1432$ and $1423$ avoiding permutations. A special case of Conjecture \ref{conj:1.3} is Gao's conjecture \ref{conj:gao}. Conjecture \ref{conj:1.3} has been verified by computer for all permutations in $S_n$ for $n\leq 8$.

\subsection{Extending Theorem \ref{thm:gen}}	
The monomials $M(x_1, \dots, x_n) = \prod_{(i, j) \in K} x_i$ we constructed from 
	diagrams $K \in \Purple_{k, \sigma_k}(D(\sigma))$   in Theorem \ref{thm:gen} do not always characterize all monomials for which 

$$\mathfrak{S}_{\sigma}(x_1, \ldots, x_n)-M(x_1, \ldots, x_n) \mathfrak{S}_{\pi}(x_{1}, \ldots, \widehat{x_k}, \ldots, x_{n}) \in \mathbb{Z}_{\geq 0}[x_1, \ldots, x_n]$$ holds; recall that   $\pi \in S_{n-1}$  is obtained by removing row $k$ and columns $\sigma_k$ of $D(\sigma)$. The following example illustrates this:

\begin{example} For the permutation $\sigma=1432$ and its pattern $\pi=132$ (coming from the subword $142$ of $1432$) obtained by removing row $k=3$ and column $\sigma_k=3$ of $D(\sigma)$, the set of monomials of the form $\prod_{(i, j) \in K} x_i$ constructed from diagrams $K \in \Purple_{3, \sigma_3}(D(\sigma))$ is $\{x_1x_3, x_2x_3\}$, yet the monomial $M(x_1, \ldots, x_n) = x_1x_2$ also yields $\mathfrak{S}_{\sigma}(x_1, \ldots, x_n)-M(x_1, \ldots, x_n) \mathfrak{S}_{\pi}(x_{1}, \ldots, \widehat{x_k}, \ldots, x_{n}) \in \mathbb{Z}_{\geq 0}[x_1, \ldots, x_n]$. In contrast, for $\sigma=1432$ and its pattern $\pi=132$ (coming from the subword $143$ of $1432$)  obtained by removing row $k=4$ and column $\sigma_4=2$ of $D(\sigma)$, the set of monomials of the form $\prod_{(i, j) \in K} x_i$ constructed from  diagrams $K \in \Purple_{4, \sigma_4}(D(\sigma))$ is $\{ x_1x_2, x_1x_3, x_2x_3\}$ and these are all the monomials for which $\mathfrak{S}_{\sigma}(x_1, \ldots, x_n)-M(x_1, \ldots, x_n) \mathfrak{S}_{\pi}(x_{1}, \ldots, \widehat{x_k}, \ldots, x_{n}) \in \mathbb{Z}_{\geq 0}[x_1, \ldots, x_n]$. \end{example}

	\begin{problem} \label{prob:char} Given permutation $\sigma \in S_n$ and its pattern  $\pi \in S_{n-1}$  obtained by removing row $k$ and column $\sigma_k$ of $D(\sigma)$, characterize all monomials $M(x_1, \ldots, x_n)$ for which  $$\mathfrak{S}_{\sigma}(x_1, \ldots, x_n)-M(x_1, \ldots, x_n) \mathfrak{S}_{\pi}(x_{1}, \ldots, \widehat{x_k}, \ldots, x_{n}) \in \mathbb{Z}_{\geq 0}[x_1, \ldots, x_n]$$ holds. \end{problem}

We conjecture that for $1432$ and $1423$ avoiding permutations Theorem \ref{thm:gen} characterizes these monomials:

\begin{conjecture} \label{conj:16} For $1432$ and $1423$ avoiding permutation $\sigma \in S_n$ and its pattern $\pi \in S_{n-1}$ obtained by removing row $k$ and column $\sigma_k$ of $D(\sigma)$, 
$$\mathfrak{S}_{\sigma}(x_1, \ldots, x_n)-M(x_1, \ldots, x_n) \mathfrak{S}_{\pi}(x_{1}, \ldots, \widehat{x_k}, \ldots, x_{n}) \in \mathbb{Z}_{\geq 0}[x_1, \ldots, x_n]$$ if and only if $M(x_1, \dots, x_n) = \prod_{(i, j) \in K} x_i$, where $K \in \Purple_{k, \sigma_k}(D(\sigma))$. 
\end{conjecture}

Conjecture \ref{conj:16} has been verified by computer for all permutations in $S_n$ for $n\leq 8$.
We note that there are permutation and pattern pairs $\sigma \in S_n$ and $\pi \in S_{n-1}$, where $\sigma$ is not $1432$ and $1423$ avoiding, yet $\mathfrak{S}_{\sigma}(x_1, \ldots, x_n)-M(x_1, \ldots, x_n) \mathfrak{S}_{\pi}(x_{1}, \ldots, \widehat{x_k}, \ldots, x_{n}) \in \mathbb{Z}_{\geq 0}[x_1, \ldots, x_n]$ if and only if $M(x_1, \dots, x_n) = \prod_{(i, j) \in K} x_i$, where $K \in \Purple_{k, \sigma_k}(D(\sigma))$. An example is $\sigma=1423$ and any of its patterns $\pi$ obtained from $D(\sigma)$ by removing row $k$ and column $\sigma_k$ ($k \in [4]$).

	\subsection{Extending Theorem \ref{thm:main}} As stated, Theorem \ref{thm:main} does not hold for  all permutations. However, it is natural to wonder about the following extension:

	 \begin{problem} \label{prob:gen}  Let $w \in S_n$ and let $u$ be a subword of $w$. Using the monomials from Theorem \ref{thm:gen} (or its extension asked for in Problem \ref{prob:char}) is it possible to pick suitable monomials $m_{w,v} \in \Z[x_1, \ldots, x_n]$ to make the expression

	$$\sum_{u \le v \le w} (-1)^{\left|w\right|-\left|v\right|} m_{w, v} \Schub_{\perm(v)}(\mathbf{x}_{w^{-1}(v)})$$ belong to $\Z_{\ge0}[x_1, \dots, x_n]$?
 	  \end{problem}

	Note that a positive answer to Problem \ref{prob:gen} would be an extension of Theorem \ref{thm:main} which would readily imply Conjecure \ref{conj:1.3} as well as Gao's Conjecture \ref{conj:gao}.

	\bibliographystyle{plain}
	\bibliography{gao-bibliography}

\begin{thebibliography}{10}

\bibitem{laddermoves}
N.~Bergeron and S.~Billey.
\newblock R{C}-graphs and {S}chubert polynomials.
\newblock {\em Experiment. Math.}, 2(4):257--269, 1993.

\bibitem{BJS}
S.~Billey, W.~Jockusch, and R.~P. Stanley.
\newblock Some combinatorial properties of {S}chubert polynomials.
\newblock {\em J. Algebraic Combin.}, 2(4):345--374, 1993.

\bibitem{bijmac}
S.~C. Billey, A.~E. Holroyd, and B.~J. Young.
\newblock A bijective proof of {M}acdonald's reduced word formula.
\newblock {\em Algebr. Comb.}, 2(2):217--248, 2019.

\bibitem{FG}
N.~Fan and P.~Guo.
\newblock Upper bounds of schubert polynomials, 2019.
\newblock {\arxiv{1909.07206}}.

\bibitem{FMS}
A.~Fink, K.~M\'{e}sz\'{a}ros, and A.~St. Dizier.
\newblock Schubert polynomials as integer point transforms of generalized
  permutahedra.
\newblock {\em Adv. Math.}, 332:465--475, 2018.

\bibitem{zeroone}
A.~Fink, K.~M\'esz\'aros, and A.~St. Dizier.
\newblock Zero-one {S}chubert polynomials.
\newblock {\em Math. Z.}, 2020.

\bibitem{FK1993}
S.~Fomin and A.~N. Kirillov.
\newblock The {Y}ang-{B}axter equation, symmetric functions, and {S}chubert
  polynomials.
\newblock {\em Discrete Math.}, 153(1):123--143, 1996.
\newblock Proceedings of the 5th Conference on Formal Power Series and
  Algebraic Combinatorics.

\bibitem{reduced}
S.~Fomin and A.~N. Kirillov.
\newblock Reduced words and plane partitions.
\newblock {\em J. Algebraic Combin.}, 6(4):311--319, 1997.

\bibitem{nilcoxeter}
S.~Fomin and R.~P. Stanley.
\newblock Schubert polynomials and the nil{C}oxeter algebra.
\newblock {\em Adv. in Math.}, 103(2):196 -- 207, 1994.

\bibitem{Gao}
Y.~Gao.
\newblock Principal specializations of schubert polynomials and pattern
  containment.
\newblock {\em European Journal of Combinatorics, to appear}, 2020.

\bibitem{june}
J.~Huh, J.~Matherne, K.~M\'esz\'aros, and A.~St.~Dizier.
\newblock Logarithmic concavity of {S}chur and related polynomials.
\newblock {\em arXiv e-prints}, 2017.
\newblock {\arxiv{1906.09633.pdf}}.

\bibitem{multidegree}
A.~Knutson and E.~Miller.
\newblock Gr\"{o}bner geometry of {S}chubert polynomials.
\newblock {\em Ann. of Math. (2)}, 161(3):1245--1318, 2005.

\bibitem{KP}
W.~Kra\'skiewicz and P.~Pragacz.
\newblock Foncteurs de {S}chubert.
\newblock {\em C. R. Acad. Sci. Paris S\'er. I Math.}, 304(9):209--211, 1987.

\bibitem{thomas}
T.~Lam, S.~Lee, and M.~Shimozono.
\newblock Back stable {S}chubert calculus, 2018.
\newblock {\arxiv{1806.11233}}.

\bibitem{LS1}
A.~Lascoux and M.-P. Sch\"utzenberger.
\newblock Polyn\^omes de {S}chubert.
\newblock {\em C. R. Acad. Sci. Paris S\'er. I Math.}, 294(13):447--450, 1982.

\bibitem{lenart}
C.~Lenart.
\newblock A unified approach to combinatorial formulas for {S}chubert
  polynomials.
\newblock {\em J. Algebraic Combin.}, 20(3):263--299, 2004.

\bibitem{macdonald}
I.~G. Macdonald.
\newblock {\em Notes on Schubert polynomials}.
\newblock Publ. LaCIM, UQAM, Montr\`eal, 1991.

\bibitem{magyar}
P.~Magyar.
\newblock Schubert polynomials and {B}ott-{S}amelson varieties.
\newblock {\em Comment. Math. Helv.}, 73(4):603--636, 1998.

\bibitem{manivel}
L.~Manivel.
\newblock {\em Symmetric functions, {S}chubert polynomials and degeneracy
  loci}, volume~6 of {\em SMF/AMS Texts and Monographs}.
\newblock American Mathematical Society, Providence, RI; Soci\'et\'e
  Math\'ematique de France, Paris, 2001.
\newblock Translated from the 1998 French original by John R. Swallow, Cours
  Sp\'ecialis\'es [Specialized Courses], 3.

\bibitem{MTY}
C.~Monical, N.~Tokcan, and A.~Yong.
\newblock Newton polytopes in algebraic combinatorics.
\newblock {\em Selecta Math. (N.S.)}, 25(66), 2019.

\bibitem{pak}
A.~H. Morales, I.~Pak, and G.~Panova.
\newblock Asymptotics of principal evaluations of {S}chubert polynomials for
  layered permutations.
\newblock {\em Proc. Amer. Math. Soc.}, 147(4):1377--1389, 2019.

\bibitem{flaggedLRrule}
V.~Reiner and M.~Shimozono.
\newblock Key polynomials and a flagged {L}ittlewood--{R}ichardson rule.
\newblock {\em J. Combin. Theory Ser. A}, 70(1):107--143, 1995.

\bibitem{she}
R.~P. Stanley.
\newblock Some schubert shenanigans.
\newblock {\em arXiv e-prints}, 2017.
\newblock {\arxiv{1704.00851.pdf}},.

\bibitem{132}
A.~Weigandt.
\newblock Schubert polynomials, 132-patterns, and {S}tanley's conjecture.
\newblock {\em Algebr. Comb.}, 1(4):415--423, 2018.

\bibitem{prismtableaux}
A.~Weigandt and A.~Yong.
\newblock The prism tableau model for {S}chubert polynomials.
\newblock {\em J. Comb. Theory, Ser. A}, 154:551--582, 2018.

\end{thebibliography}
\end{document}